\documentclass[11pt,reqno]{amsart}
\usepackage[a4paper,bindingoffset=0.1in,left=1.2in,right=1.3in,top=1.4in,bottom=1.5in,footskip=.25in]{geometry}
\usepackage{indentfirst,amssymb,amsmath,amsthm}    
\usepackage{newtxtext,newtxmath}
\usepackage{setspace}
\usepackage{times}
\usepackage[utf8]{inputenc}
\usepackage[T1]{fontenc}
\usepackage{verbatim}
\usepackage{hyperref}
\usepackage{csquotes}
\hypersetup{colorlinks=true, linkcolor=blue,citecolor=blue, urlcolor=blue}
\newcommand{\divides}{\mid}

\DeclareMathOperator{\ord}{ord}
\DeclareMathOperator{\dime}{dim}
\DeclareMathOperator{\pideg}{PI-deg}

\DeclareMathOperator{\gcdi}{gcd}
\DeclareMathOperator{\lcmu}{lcm}

\DeclareMathOperator{\tors}{tor}

\DeclareMathOperator{\ran}{rank}
\DeclareMathOperator{\mo}{mod}
\usepackage{mathtools}

\numberwithin{equation}{section}
 
\newtheorem{theo}{Theorem}[section]

\newtheorem{lemm}[theo]{Lemma}
\newtheorem{rema}[theo]{Remark}
\newtheorem{coro}[theo]{Corollary}
\newtheorem{prop}[theo]{Proposition}

\title{On Simple Modules over the Quantum Matrix algebra at roots of unity}
\author{Sanu Bera \ \ Snehashis Mukherjee}

\begin{document}

\maketitle
\begin{abstract}
This article investigates the two-parameter quantum matrix algebra at roots of unity. 
In the roots of unity setting, this algebra becomes a Polynomial Identity (PI) algebra and it is known that simple modules over such algebra are finite-dimensional with dimension at most the PI degree. We determine the center, compute the PI degree, and classify simple modules for two-parameter quantum matrix algebra, up to isomorphism, over an algebraically closed field of arbitrary characteristics.
\end{abstract}
\section{{Introduction}}
Let $\mathbb{K}$ be a field and $\mathbb{K}^*$ denote the multiplicative group $\mathbb{K}\setminus\{0\}$. Let $\alpha,\beta\in \mathbb{K}^*$ be two quantum parameters. The two-parameter quantum matrix algebra $M_n(\alpha,\beta)$ is an associative algebra over $\mathbb{K}$ generated by $\{X_{ij}\}_{1\leq i,j\leq n}$ subject to the defining relations
\begin{align*}
X_{ij}X_{ik}&=\alpha X_{ik}X_{ij} \hspace{.3cm} \text{if}\  k<j\\
X_{ij}X_{kj}&=\beta X_{kj}X_{ij}\hspace{.3cm} \text{if}\ k<i\\
X_{ij}X_{st}&=\beta\alpha^{-1}X_{st}X_{ij} \hspace{.3cm} \text{if}\ s<i,\ j<t\\
X_{ij}X_{st}-X_{st}X_{ij}&=(\beta-\alpha^{-1}) X_{sj}X_{it} \hspace{.3cm} \text{if}\ s<i,\ t<j.
\end{align*}
This algebra was introduced in 1990 by Takeuchi \cite{take} to construct a two-parameter quantization of the algebraic group $GL_n$ of invertible $n\times n$ matrices. Note that when $\alpha=\beta=1$, this algebra becomes the coordinate ring of the variety of $n\times n$ matrices and thus we can view it as a deformation of this coordinate ring. For $q\in \mathbb{K}^*$, we can recover the single parameter quantum matrix algebras $\mathcal{O}_q(M_n)$ introduced by Faddeev, Reshetikhin and Takhtajan \cite{manin, par}, and $\mathcal{M}at_n(q)$ introduced by Dipper and Donkin \cite{dip}, from the above definition as a special case by setting $(q,q)$ and $(1,q)$ as $(\alpha,\beta)$, respectively. It is worth noting that since then, various authors have introduced more multi-parameter quantizations of general linear groups \cite{res, sud}, with Takeuchi's two-parameter group being a special case of these. It is well-known that $M_n(\alpha,\beta)$ is Artin-Schelter regular of global and GK dimension $n^2$ (cf. \cite[Theorem 2]{ast}). Moreover $M_n(\alpha,\beta)$ may be presented as an iterated skew polynomial algebra twisted by automorphisms and derivations (cf. \cite[Theorem I.2.7]{brg}). These single-parameter quantum matrix algebras have been studied extensively for automorphism group \cite{ac, my,gl}, isomorphism problem \cite{g}, center \cite{nym,jkjz,ddjj}, prime and primitive spectrum \cite{gole,gc, ll, bln}, and representations \cite{dj,cj,gls,sbsm} in both generic and roots of unity setting. In this article, we wish to study the two-parameter quantum matrix algebra $M_2(\alpha,\beta)$ in which the deformation parameters are roots of unity. 
\par For a root of unity $\xi\in\mathbb{K}^*$, we write $\ord(\xi)$ to denote the order of $\xi$ as an element in the group $\mathbb{K}^*$, i.e., $\ord(\xi):=\min\{k\in\mathbb{Z}_{+}:\xi^k=1\}$. In the roots of unity setting, the algebra $M_n(\alpha,\beta)$ becomes a finitely generated module over its center and thus a Polynomial Identity (PI) algebra. In \cite{gl}, Gaddis and Lamkin explored PI quantum matrix algebras $M_n(\alpha,\beta)$ and their automorphisms using the noncommutative discriminant. For the two-parameter case with $n=2$, they demonstrated that all automorphisms are graded when the center is a polynomial ring, given the assumption that $\ord(\alpha)$ and $\ord(\beta)$ are relatively prime. Importantly, we establish that the center of $M_2(\alpha,\beta)$ can be a polynomial algebra without assuming the orders are relatively prime, see Corollary \ref{cenpoly}.
\par The theory of PI algebras serves as a pivotal tool in analyzing algebraic structure of $M_n(\alpha,\beta)$, with the PI degree standing out as a fundamental invariant. In the case of a prime affine PI algebra over an algebraically closed field, the PI degree bounds the $\mathbb{K}$-dimension of each simple module \cite[Theorem I.13.5]{brg} and this bound is attained by the algebra \cite[Lemma III.1.2]{brg}. Since the PI degree provides insight into simple $M_2(\alpha,\beta)$-modules, we will determine the PI degree of $M_2(\alpha,\beta)$, as shown in Theorem \ref{pithm}. 
\par In \cite{kar}, Karimipour explores the irreducible representations of the quantum matrix algebra $M_{\mathrm{q},\mathrm{p}}(2)$ with deformation parameters $\mathrm{p}$ and $\mathrm{q}$, where $\alpha=(\mathrm{pq})^{-1}$ and $\beta=\mathrm{p}\mathrm{q}^{-1}$ in our context. It was established in \cite{kar} that if $\ord(\mathrm{p})=\ord(\mathrm{q})=r$ (say) is odd, the so-called \enquote{toroidal}, \enquote{semitoroidal} and \enquote{cylindrical} representations are irreducible representations of dimension $r^2$. However, our article argues that this result is not generally true when $\mathrm{p}$ and $\mathrm{q}$ are two distinct roots of unity (see Corollary \ref{corogcd}). Consequently, the classification of simple modules over this algebra remains incomplete.  In this article, we continue the study of simple modules over $M_2(\alpha,\beta)$ and aim to provide a complete classification assuming $\alpha$ and $\beta$ are roots of unity. 
\par Throughout the article, all modules under consideration are the right modules, $\mathbb{K}$ is an algebraically closed field of arbitrary characteristic. The article is organized as follows. Section \ref{secpre} presents key findings concerning the algebra $M_2(\alpha,\beta)$ and delves into preliminary results regarding Polynomial Identity algebras. In Section \ref{seccen}, we determine the center of the algebra $M_2(\alpha,\beta)$ under the assumption that $\ord(\alpha\beta)=\ord(\alpha\beta^{-1})$, which includes the previous study by Gaddis where $\ord(\alpha)$ and $\ord(\beta)$ are relatively prime. Section \ref{secpideg} addresses the computation of PI degree for quantum matrix algebra $M_2(\alpha,\beta)$. Our method leverages the derivation erasing process with the De Concini-Processi algorithm to determine the PI degree. In Section \ref{seccon}, we construct three non-isomorphic simple $M_2(\alpha,\beta)$-modules: $\mathscr{V}_1(\underline{\mu}),\mathscr{V}_2(\underline{\mu})$ and $\mathscr{V}_3(\underline{\mu})$. In Section \ref{seccla} we classify simple modules over $M_2(\alpha,\beta)$ by considering the actions of certain normal or central elements and then prove that any $X_{12}, X_{21}$ and $D$-torsionfree simple $M_2(\alpha,\beta)$-module is isomorphic to one of the above three simple modules and consequently provide a necessary and sufficient condition for a module to be maximal dimensional. Finally, in Section \ref{seciso} we determine the criteria under which two simple modules in the same class are isomorphic, thereby completing the classification problem for simple modules over $M_2(\alpha,\beta)$.
\section{Preliminaries}\label{secpre}
\subsection{Torsion and Torsionfree Modules} Let $A$ be a $\mathbb{K}$-algebra and $M$ be a right $A$-module and $S\subset A$ be a right Ore set. The submodule \[\tors_{S}(M)=\{m\in M:ms=0\ \text{for some}\ s\in S\}\] is called the $S$-torsion submodule of $M$. The module $M$ is said to be $S$-torsion if $\tors_{S}(M)=M$ and $S$-torsionfree if $\tors_{S}(M)=0$. If Ore set $S$ is generated by a single element $x\in A$, then we say that the $S$-torsion/torsionfree module $M$ is $x$-torsion/torsionfree.
\par A nonzero element $x$ of an algebra $A$ is called normal element if $xA=Ax$. Observe that if $x\in A$ is a normal element, then the set $S=\{x^k:k\geq 0\}$ is an Ore set generated by $x$. The following lemma is obvious.
\begin{lemm}\label{normalaction}
    Suppose that $A$ is an algebra, $x\in A$ is a normal element and $M$ is a simple $A$-module. Then either $Mx=0$ (if $M$ is $x$-torsion) or the map $x_{M}:M\rightarrow M$ given by $m\mapsto mx$ is an isomorphism (if $M$ is $x$-torsionfree). 
\end{lemm}
\subsection{Skew Polynomial Presentation} Recall that the algebra $M_2(\alpha,\beta)$ is an $\mathbb{K}$-algebra generated by $X_{11},X_{12},X_{21}$ and $X_{22}$ together with the relations
\begin{align*}
X_{12}X_{11}&=\alpha X_{11}X_{12} & X_{22}X_{21}&=\alpha X_{21}X_{22} \\
X_{21}X_{11}&=\beta X_{11}X_{21} & X_{22}X_{12}&=\beta X_{12}X_{22}\\
X_{21}X_{12}&=\beta \alpha^{-1}X_{12}X_{21}& X_{22}X_{11}-X_{11}X_{22}&=(\beta-\alpha^{-1})X_{12}X_{21}
\end{align*}
The algebra $M_2(\alpha,\beta)$ has an iterated skew polynomial presentation to the ordering of the variables $X_{11},X_{12},X_{21}$ and $X_{22}$ of the form
\begin{equation}\label{spp}
    \mathbb{K}[X_{11}][X_{12},\sigma_{12}][X_{21},\sigma_{21}][X_{22},\sigma_{22},\delta_{22}]
\end{equation}
where the $\sigma_{12},\sigma_{21}$ and $\sigma_{22}$ are $\mathbb{K}$-linear automorphisms and the $\delta_{22}$ is $\mathbb{K}$-linear $\sigma_{22}$-derivation such that
\begin{align*}
    \sigma_{12}(X_{11})&=\alpha X_{11}& \sigma_{21}(X_{11})&=\beta X_{11}& \sigma_{21}(X_{12})&=\beta \alpha^{-1}X_{12}\\
    \sigma_{22}(X_{11})&=X_{11}& \sigma_{22}(X_{12})&=\beta X_{12}& \sigma_{22}(X_{21})&=\alpha X_{21}\\
    \delta_{22}(X_{11})&=(\beta-\alpha^{-1})X_{12}X_{21}& \delta_{22}(X_{12})&=0& \delta_{22}(X_{21})&=0.
\end{align*}
Thus the algebra $M_2(\alpha,\beta)$ is a prime affine noetherian domain and the family of ordered monomials $\{X_{11}^{a}X_{12}^{b}X_{21}^{c}X_{22}^{d}: a,b,c,d\geq 0\}$ forms a $\mathbb{K}$-basis. The defining relations ensure that the elements $X_{12}$ and $X_{21}$ are normal. The quantum determinant $D$ is defined by \[D:=X_{11}X_{22}-\alpha^{-1}X_{12}X_{21}=X_{22}X_{11}-\beta X_{12}X_{21}.\] Now we can verify the following commutation relations
\[DX_{11}=X_{11}D,\ DX_{12}=\alpha^{-1}\beta X_{12}D,\ DX_{21}=\alpha\beta^{-1} X_{21}D,\ DX_{22}=X_{22}D.\]
Therefore $D$ is also normal element in $M_2(\alpha,\beta)$. These normal elements will play a crucial role in classifying simple modules. Note that $D$ is central in the single parameter case where $\alpha=\beta$.
\begin{lemm}\label{crelation}
    The following identities hold in $M_2(\alpha,\beta)$:
    \begin{itemize}
        \item [(1)] $X_{22}^kX_{11}=X_{11}X_{22}^k+\alpha^{-1}[(\alpha\beta)^k-1]X_{12}X_{21}X_{22}^{k-1}$.
        \item [(2)] $X_{22}X_{11}^k=X_{11}^kX_{22}+\beta[1-(\alpha\beta)^{-k}]X_{12}X_{21}X_{11}^{k-1}$.
    \end{itemize}
\end{lemm}
\begin{proof}
The assertions follow from the defining relations $M_2(\alpha,\beta)$ by using induction on $k$.
\end{proof}
\begin{coro}\label{co}
    If $\alpha$ and $\beta$ are $l$-th roots of unity, then $X_{11}^l,X_{12}^{l},X_{21}^{l}$ and $X_{22}^{l}$ are in the center of $M_2(\alpha,\beta)$.
\end{coro}
\subsection{Polynomial Identity Algebras} In the roots of unity setting, the quantum matrix algebra $M_2(\alpha,\beta)$ becomes a finitely generated module over its center, by Corollary \ref{co}. Hence as a result \cite[Corollary 13.1.13]{mcr}, the algebra $M_2(\alpha,\beta)$ becomes a polynomial identity algebra. This sufficient condition on the parameters to be PI algebra is also necessary.
\begin{prop}\emph{(\cite[Proposition 2.1]{gl})}
    The algebra $M_2(\alpha,\beta)$ is a PI algebra if and only if $\alpha$ and $\beta$ are roots of unity.
\end{prop}
Kaplansky's Theorem has a striking consequence in the case of a prime affine PI algebra over an algebraically closed field.
\begin{prop}\label{pidimresult}\emph{(\cite[Theorem I.13.5]{brg})}
Let $A$ be a prime affine PI algebra over an algebraically closed field $\mathbb{K}$ and $V$ be a simple $A$-module. Then $V$ is a finite-dimensional vector space over $\mathbb{K}$ with $\dime_{\mathbb{K}}(V)\leq \pideg (A)$.
\end{prop} 
This result provides the important link between the PI degree of a prime affine PI algebra over an algebraically closed field and its irreducible representations. Moreover, the upper bound PI-deg($A$) is attained for such an algebra $A$ (cf. \cite[Lemma III.1.2]{brg}). 
\begin{rema}\label{fdb}
The algebra $M_2(\alpha,\beta)$ is classified as a prime affine PI algebra. From the above discussion, it is quite clear that each simple $M_2(\alpha,\beta)$-module is finite-dimensional and can have dimension at most $\pideg {M_2(\alpha,\beta)}$. Therefore the calculation of PI degree for $M_2(\alpha,\beta)$ is of substantial importance. Section \ref{secpideg} will focus on determining the PI degree of $M_2(\alpha,\beta)$. All the aforementioned results for the algebra $M_2(\alpha,\beta)$ are also applicable to $M_n(\alpha,\beta)$ in general (cf. \cite{take,gl,brg}). 
\end{rema}
\section{The Center of $M_2(\alpha,\beta)$}\label{seccen}
This section aims to calculate the center of the two-parameter quantum matrix algebra $M_2(\alpha,\beta)$ in the root of unity setting. The center for the single-parameter case where $\alpha=\beta$ is the $\mathbb{K}$-algebra generated by $X_{11}^m, X_{22}^m, D, X_{11}^{m'}X_{22}^{m'}$ and $X_{12}^rX_{21}^{m-r}$, where $r=0,1,\cdots, m$ and $m=\ord(\alpha), m'=\ord(\alpha^2)$, was determined in \cite{jkjz}. When $\alpha=\beta^{-1}$, the algebra $M_2(\alpha,\alpha^{-1})$ becomes a single parameter quantum affine space (see Subsection \ref{piisoqa}) and the center is the $\mathbb{K}$ algebra generated by $X_{ij}^m,X^{m-1}_{11}X_{22}$ and $X_{11}^{m-2}X_{12}X_{21}$, where $m=\ord(\alpha)$ and $i,j=1,2$, by using \cite[Proposition 7.1]{di}. 
\par In the remainder of this section, we assume that $\alpha,\beta\in \mathbb{K}^*$ with $\alpha\beta^{\pm 1}\neq 1$ such that $m:=\ord(\alpha)$ and $n:=\ord(\beta)$ are finite. Denote $l:=\lcmu(m,n)$. It is easy to check using the defining relations that the elements $X_{11}^l,\ X_{12}^l,\ X_{21}^l,\ X_{22}^l$ and $D^{\ord(\alpha\beta^{-1})}$ are in the center of $M_2(\alpha,\beta)$. Denote the lexicographical order on $\mathbb{Z}^2_{+}$ by $(a',b')<(a,b)$, that means $a'<a$, or $a'=a$ and $b'<b$. For any $0\neq f\in M_2(\alpha,\beta)$, we may uniquely write
\[f=f_{uv}X_{11}^{u}X_{22}^v+\sum\limits_{(a,b)<(u,v)}f_{ab}X_{11}^{a}X_{22}^b,\]
where each $f_{uv}\neq 0$ and $f_{ab}$ is an element of the subalgebra generated by $X_{12}$ and $X_{21}$. Denote the degree of $f$ by $\deg(f)=(u,v)$. 
\par In the following, we wish to determine the center of $M_2(\alpha,\beta)$ completely under the assumption $\ord(\alpha\beta)=\ord(\alpha\beta^{-1})$. Importantly, this equality gives $\ord(\alpha\beta)=\ord(\alpha\beta^{-1})$ is divisible by $\lcmu(\ord(\alpha^2),\ord(\beta^2))$. In particular, this equality holds if $m=\ord(\alpha)$ and $n=\ord(\beta)$ are relatively prime, and in this case, the center was determined in \cite[Lemma 3.2]{gl}, which becomes a polynomial algebra. Moreover the equality $\ord(\alpha\beta)=\ord(\alpha\beta^{-1})$ holds in a more general context as mentioned in \cite[Proposition 8.2]{bmn}. 
\begin{lemm}\label{implemma}
    Let $\alpha,\beta\in\mathbb{K}^*$ with $\alpha \beta^{\pm 1}\neq 1$ such that $\ord(\alpha\beta)=\ord(\alpha\beta^{-1})=t~(\text{say})$. Then each non-negative solution $(a,b)$ of $\alpha^a\beta^b=1=\alpha^b\beta^a$ satisfies \[a\equiv k t~(\mo l)\ \ \text{and}\ \ b\equiv k t~(\mo l)\] for some $0\leq k\leq \frac{l}{t}-1$.
\end{lemm}
\begin{proof}
Suppose $(a,b)$ is a nonnegative solution of $\alpha^a\beta^b=1=\alpha^b\beta^a$. Then $\ord(\alpha\beta^{-1})$ divides $(a-b)$ and $\ord(\alpha\beta)$ divides $(a+b)$. This implies that $t$ divides $2a$ and $2b$, according to our assumption.
\par We first claim that $t$ divides both $a$ and $b$. If $t$ is odd, then we are done. If $t$ is even, we can write
$a=x\frac{t}{2}$ and $b=y\frac{t}{2}$ for some integers $x,y$ with $x\pm y$ is even. Note that $x+y$ or $x-y$ is divisible by $4$. Without loss of generality, we may assume that $x-y$ is divisible by 4. Now using the fact $t$ is divisible by $\lcmu(\ord(\alpha^2),\ord(\beta^2))$, simplify the equality
\[1=\alpha^a\beta^b=(\alpha\beta)^{x\frac{t}{2}}\beta^{(y-x)\frac{t}{2}}=(\alpha\beta)^{x\frac{t}{2}}\]
to obtain $t$ divides $x\frac{t}{2}$. This implies $x$ must be even. Similarly, the equality $1=\alpha^b\beta^a$ gives $y$ is even. Thus we can conclude that $t$ divides $a$ and $b$.
\par Suppose $a\equiv a^*t~(\mo l)$ and $b\equiv b^* t~(\mo l)$ for some $0\leq a^*,b^*\leq \frac{l}{t}-1$. We now claim that $a^*=b^*$. Indeed the equality $\alpha^a\beta^b=1$ gives
\[ \alpha^{a^*t}\beta^{a^*t}=1
    \implies \alpha^{(a^*-b^*)t}(\alpha\beta)^{b^*t}=1
    \implies \ord(\alpha)| (a^*-b^*)t.\]
    Similarly the equality $\alpha^b\beta^a=1$ gives $\ord(\beta)$ divides $(a^*-b^*)t$. Thus combining these two we obtain 
    \[l|(a^*-b^*)t\implies \frac{l}{t}|(a^*-b^*)\implies a^*=b^*.\]
    Hence we can conclude that $a\equiv kt~(\mo l)$ and $b\equiv kt~(\mo l)$ for some $0\leq k\leq \frac{l}{t}-1$. This completes the proof.
\end{proof}
\begin{theo}\label{centhm}
 Let $\alpha,\beta\in\mathbb{K}^*$ with $\alpha\beta^{\pm 1}\neq 1$ such that $\ord(\alpha\beta)=\ord(\alpha\beta^{-1})=t~(\text{say})$. Then the center of $M_2(\alpha,\beta)$ is the $\mathbb{K}$-algebra generated by 
 \begin{equation}\label{cengen}
     X_{11}^l,\ X_{12}^l,\ X_{21}^l,\  X_{22}^l,\ X_{11}^{t}X_{22}^{t}\ \text{and}\ X_{12}^{t}X_{21}^{t}.
 \end{equation}
\end{theo}
\begin{proof}
Let $\mathcal{Z}$ denote the center $M_2(\alpha,\beta)$ and $\mathcal{S}$ denote the $\mathbb{K}$-subalgebra generated by the elements in (\ref{cengen}). Suppose 
\[f=f_{uv}X_{11}^{u}X_{22}^v+\sum\limits_{(a,b)<(u,v)}f_{ab}X_{11}^{a}X_{22}^b\]
is a nonzero element in $\mathcal{Z}\setminus \mathcal{S}$ with the minimal degree $(u,v)$. Now we can simplify the equality $X_{11}f=fX_{11}$ to obtain
\[0=X_{11}f-fX_{11}=(X_{11}f_{uv}-f_{uv}X_{11})X_{11}^uX_{22}^v+\ \ \text{lower\ degree terms}.\] 
This implies $X_{11}f_{uv}-f_{uv}X_{11}=0$. Similarly the equality $X_{22}f=fX_{22}$ implies that $X_{22}f_{uv}-f_{uv}X_{22}=0$. Write $f_{uv}=\sum \gamma_{ij}X_{12}^iX_{21}^j$ with $\gamma_{ij}\in \mathbb{K}$. Then 
\[0=X_{11}f_{uv}-f_{uv}X_{11}=\sum \gamma_{ij}(\alpha^{-i}\beta^{-j}-1)X_{12}^iX_{21}^jX_{11},\]
\[0=X_{22}f_{uv}-f_{uv}X_{22}=\sum \gamma_{ij}(\beta^{i}\alpha^{j}-1)X_{12}^iX_{21}^jX_{22}.\]
This gives $\alpha^{-i}\beta^{-j}=1=\beta^{i}\alpha^{j}$ for all $(i,j)$ in support of $f_{uv}$. Therefore by Lemma \ref{implemma} we obtain $i\equiv k_{ij}t~(\mo l)$ and $j\equiv k_{ij}t~(\mo l)$ where $0\leq k_{ij}\leq \frac{l}{t}-1$, for all $(i,j)$ in support of $f_{uv}$. Thus $f_{uv}$ is a polynomial in $X_{12}^{t}X_{21}^{t},\ X_{12}^{l}$ and $X_{21}^l$. Then we simplify the equalities $X_{12}f=fX_{12}$ and $X_{21}f=fX_{21}$ to obtain $\beta^v\alpha^{-u}=1$ and $\alpha^v\beta^{-u}=1$ respectively. Again by Lemma \ref{implemma}, we have $u\equiv kt~(\mo l)$ and $v\equiv kt~(\mo l)$. So $f-f_{uv}X_{11}^{u}X_{22}^v$ is a nonzero element in $\mathcal{Z}\setminus \mathcal{S}$, with the degree less than $(u,v)$. This contradicts the fact that the degree $(u,v)$ of $f$ is minimal. This completes the proof.
\end{proof}
\par Now we establish an algebraic dependence relationship for the central element $D^{t}$, where $t=\ord(\alpha\beta^{-1})=\ord(\alpha\beta)$. First we can simplify the the $k$-th power of $D=X_{11}X_{22}-\alpha^{-1}X_{12}X_{21}$ using the following fact:
\begin{enumerate}
    \item The element $X_{12}X_{21}$ commute with $X_{11}X_{22}$.
    \item For any $r\geq 1$, using Lemma \ref{crelation}, we get \[(X_{11}^rX_{22}^r)(X_{11}X_{22})=X_{11}^{r+1}X_{22}^{r+1}+\alpha^{-1}[1-(\alpha\beta)^{-r}]X_{12}X_{21}X_{11}^{r}X_{22}^r.\]
\end{enumerate}
Finally, we obtain the following identity, for any $k\geq 1$:
\begin{equation}\label{didentity}
D^k=X_{11}^kX_{22}^k+\sum\limits_{r=1}^{k-1}C_{r}^{(k)}X_{11}^rX_{22}^r+(-1)^{k}\alpha^{\frac{-k(k+1)}{2}}\beta^{\frac{k(k-1)}{2}}X_{12}^kX_{21}^k 
\end{equation}
where each $C_r^{(k)}$ is an element in the subalgebra generated by $X_{12}$ and $X_{21}$. Moreover rearranging the equality (\ref{didentity}) for $k=t$, we have
\[D^{t}-X_{11}^{t}X_{22}^{t}-(-1)^{t}\alpha^{\frac{-t(t+1)}{2}}\beta^{\frac{t(t-1)}{2}}X_{12}^{t}X_{21}^{t}=\sum\limits_{r=1}^{t-1}C_{r}^{(t)}X_{11}^rX_{22}^r.\]
The left-hand side is central, and so is the right-hand side. By Theorem \ref{centhm}, we can conclude that $C_r^{(t)}=0$ for each $1\leq r\leq t-1$. Thus we obtain an algebraic dependency relation $D^{t}=X_{11}^{t}X_{22}^{t}+(-1)^{t}\alpha^{\frac{-t(t+1)}{2}}\beta^{\frac{t(t-1)}{2}}X_{12}^{t}X_{21}^{t}$ where $t=\ord(\alpha\beta^{-1})=\ord(\alpha\beta)$.
\begin{coro}\label{cenpoly}
Assume that $\ord(\alpha\beta)=\ord(\alpha\beta^{-1})=l$. This equality holds when $m=n$ is an odd prime or when $|e_m(p)-e_n(p)|\geq 1$ for all primes $p$ in the prime factorization of $m$ or $n$, where $e_k(p)$ denotes the exponent of the prime $p$ in the prime factorization of $k$ (cf. \cite[Proposition 8.2]{bmn}). In this case, according to Theorem \ref{centhm}, the center of $M_2(\alpha,\beta)$ is the polynomial algebra generated by $X_{11}^l, X_{12}^l, X_{21}^l$ and $X_{22}^l$. In particular, this assumption holds when $\gcd(m,n)=1$; in that case, the center was determined in \cite{ddjj,gl}.
\end{coro}
\section{PI Degree for $M_2(\alpha,\beta)$}\label{secpideg}
In this section, we wish to compute an explicit expression of the PI degree for $M_2(\alpha,\beta)$ at roots of unity. We will employ the derivation erasing process independent of characteristics, as established by Leroy and Matczuk \cite[Theorem 7]{lm2}. Furthermore, we will utilize a key technique for calculating the PI degree of quantum affine space as introduced by De Concini and Procesi \cite[Proposition 7.1]{di}. Given a multiplicatively antisymmetric matrix $\Lambda=(\lambda_{ij})$, the quantum affine space is the $\mathbb{K}$-algebra $\mathcal{O}_{\Lambda}(\mathbb{K}^n)$ generated by the variables $x_1,\cdots ,x_n$ subject to the relations \[x_ix_j=\lambda_{ij}x_jx_i, \ \ \ \forall\ \ \ 1 \leq i,j\leq n.\]
The algebra $\mathcal{O}_{\Lambda}(\mathbb{K}^n)$ becomes a PI algebra at roots of unity (cf. \cite[Proposition I.14.2]{brg}). The following result simplifies the key algorithm \cite[Proposition 7.1]{di} through the properties of the integral matrix associated with $\Lambda$.
\begin{prop}\emph{(\cite[Lemma 5.7]{ar})}\label{mainpi}
Suppose that $\lambda_{ij}=q^{h_{ij}}$ where $q \in \mathbb{K}^*$ is a primitive $l$-th root of unity and $H=(h_{ij})$ be a skew-symmetric integral matrix with $\ran(H)=2s$ and invariant factors $h_1,\cdots,h_s$ (come in pairs). Then PI degree of $\mathcal{O}_{\Lambda}(\mathbb{K}^n)$ is given as \[\pideg\mathcal{O}_{\Lambda}(\mathbb{K}^n)=\prod_{i=1}^{s}\frac{l}{\gcdi(h_i,l)}.\]
\end{prop}
Assume that $\alpha$ and $\beta$ are primitive $m$-th and $n$-th roots of unity, respectively. In the following, we explicitly determine the PI degree of $M_2(\alpha,\beta)$ depending on the nature of $\alpha$ and $\beta$.
\subsection{PI degree of $M_2(\alpha,\beta)$ when $\alpha\beta=1$}\label{piisoqa} In this case, by the defining relations, the algebra $M_2(\alpha,\alpha^{-1})$ becomes a quantum affine space $\mathcal{O}_{\Lambda}(\mathbb{K}^4)$ where the $(4\times 4)$-matrix $\Lambda$ of relation is
\[\Lambda:=\begin{pmatrix}
  1& \alpha^{-1}&\alpha&1\\
  \alpha&1&\alpha^2&\alpha\\
  \alpha^{-1}&\alpha^{-2}&1&\alpha^{-1}\\
  1&\alpha^{-1}&\alpha&1
\end{pmatrix}.\]
It is easily verified that the skew-symmetric integral matrix associated with $\Lambda$ has the rank $2$ and has a kernel of dimension $2$ and one pair of invariant factors $\ h_1=1$. Thus by Proposition \ref{mainpi}, the PI degree of $M_2(\alpha,\alpha^{-1})$ is $m=\ord(\alpha)$. 
\subsection{PI degree of $M_2(\alpha,\beta)$ when $\alpha\beta\neq 1$}\label{subsectionpi} 
The algebra $M_2(\alpha,\beta)$ has an iterated skew polynomial presentation 
\[\mathbb{K}[X_{11}][X_{12},\sigma_{12}][X_{21},\sigma_{21}][X_{22},\sigma_{22},\delta_{22}]\]
twisted by automorphisms and derivations as mentioned in (\ref{spp}). Now we can simplify the following
\[\delta_{22}(\sigma_{22}(X_{11}))=\delta_{22}(X_{11})=(\beta-\alpha^{-1})X_{12}X_{21}=\alpha \beta\sigma_{22}(\delta_{22}(X_{11})).\]
This holds trivially if $X_{11}$ is replaced by $X_{12}$ or $X_{21}$. Therefore the pair $(\sigma_{22},\delta_{22})$ is a $\alpha\beta$-skew derivation on $\mathbb{K}[X_{11}][X_{12},\sigma_{12}][X_{21},\sigma_{21}]$ where $\alpha\beta\neq 1$. Moreover, we can verify that all the hypotheses of the derivation erasing process in \cite[Theorem 7]{lm2} are satisfied by the skew polynomial presentation of the PI algebra $M_2(\alpha,\beta)$ when $\alpha\beta\neq 1$. Thus it follows that
\begin{equation}\label{fpi}
    \pideg M_2(\alpha,\beta)=\pideg \mathcal{O}_{\Lambda}(\mathbb{K}^4)
\end{equation}
where the $(4\times 4)$-matrix $\Lambda$ of relation is
\[\Lambda:=\begin{pmatrix}
  1& \alpha^{-1}&\beta^{-1}&1\\
  \alpha&1&\beta^{-1}\alpha&\beta^{-1}\\
  \beta&\beta\alpha^{-1}&1&\alpha^{-1}\\
  1&\beta&\alpha&1
\end{pmatrix}.\]
Let us first consider the single-parameter case when $\beta=\alpha$. In this case, the skew-symmetric integral matrix associated with $\Lambda$ has the rank $2$ and one pair of invariant factors $h_1=1$. So by Proposition \ref{mainpi} and equality (\ref{fpi}), the PI degree of $M_2(\alpha,\alpha)$ becomes $m=\ord(\alpha)$. The PI degree for this single parameter case was computed over a field of characteristic zero in \cite{jkjz}.
\par Let us assume $\alpha\neq \beta^{\pm 1}$ in remainder of this section. Suppose $\Gamma$ is the multiplicative subgroup of $\mathbb{K}^*$ generated by $\alpha$ and $\beta$. Then $\Gamma$ becomes a cyclic group of order $l=\lcmu(m,n)$ with a generator $g$, say. Now we can choose $s_1=\frac{n}{\gcdi(m,n)}$ and $s_2=\frac{m}{\gcdi(m,n)}$ such that $\langle \alpha\rangle=\langle g^{s_1}\rangle$ and $\langle \beta\rangle=\langle g^{s_2}\rangle$. Therefore there exist integers $k_1$ and $k_2$ 
with $\gcdi(k_1,m)=\gcdi(k_2,n)=1$ such that 
\begin{equation}\label{cypq}
    \alpha=g^{s_1k_1}\ \ \text{and}\ \ \beta=g^{s_2k_2}.
\end{equation}
The skew-symmetric integral matrix associated with $\Lambda$ is
\[H:=\begin{pmatrix}
    0&-s_1k_1&-s_2k_2&0\\
    s_1k_1&0&s_1k_1-s_2k_2&-s_2k_2\\
    s_2k_2&-s_1k_1+s_2k_2&0&-s_1k_1\\
    0&s_2k_2&s_1k_1&0
\end{pmatrix}.\]
The determinant of $H$ is $\left((s_1k_1)^2-(s_2k_2)^2\right)^2$, which is nonzero as $\alpha\neq \beta^{\pm 1}$. So $\ran(H)=4$. Suppose $h_1$ and $h_2$ are invariant factors for $H$. In our context, applying Proposition \ref{mainpi} we have
\begin{equation}\label{pieq}
    \pideg M_2(\alpha,\beta)=\pideg \mathcal{O}_{\Lambda}(\mathbb{K}^4)=\frac{l}{\gcdi(h_1,l)}\times \frac{l}{\gcdi(h_2,l)}.  
\end{equation}
Now from the relations between the invariant factors and determinantal divisors, we have
\begin{align*}
h_1&=\text{first~determinantal~divisor~of~$H$}=\gcdi(s_1k_1,s_2k_2)\ \ \text{and}\\
h_2&=\displaystyle\frac{\text{fourth~determinantal~divisor~of~$H$}}{\text{third~determinantal~divisor~of~$H$}}\\
&=\displaystyle\frac{[(s_1k_1)^2-(s_2k_2)^2]^2}{\gcdi\big(s_1k_1[(s_1k_1)^2-(s_2k_2)^2],[(s_1k_1)^2-(s_2k_2)^2]s_2k_2\big)}\\&=\displaystyle\frac{(s_1k_1)^2-(s_2k_2)^2}{\gcdi\big(s_1k_1,s_2k_2\big)}.
\end{align*}
Note that $h_1h_2=(s_1k_1)^2-(s_2k_2)^2$. We will use the expressions $h_1$ and $h_2$ in the following. 
\par We now claim that $\gcdi(h_1,l)=1$. In fact, suppose $\gcdi(h_1,l)=d$. Then $d$ divides $s_1k_1,s_2k_2$ and $l$. Now we compute
\[\alpha^{\frac{l}{d}}=g^{\frac{s_1k_1l}{d}}=1\ \ \text{and}\ \ \beta^{\frac{l}{d}}=g^{\frac{s_2k_2l}{d}}=1.\] This implies that $m$ and $n$ both divide $\frac{l}{d}$ and hence $l=\lcmu(m,n)$ divides $\frac{l}{d}$. Therefore we obtain $\gcdi(h_1,l)=1$.
\par Under the above setting, we obtain that
\begin{equation}\label{sif}
\gcdi(h_2,l)=\gcdi\left(h_1h_2,l\right)=\gcdi\left((s_1k_1+s_2k_2)(s_1k_1-s_2k_2),l\right)    
\end{equation}
Since $h_1=\gcdi(s_1k_1,s_2k_2),$ therefore 
\begin{equation}\label{2nd}
    \gcdi(s_1k_1+s_2k_2,s_1k_1-s_2k_2)=h_1\ \text{or}\ 2h_1
\end{equation} 
Now we can consider the following cases depending on the nature of $m$ and $n$. For any positive integer $k$, let $e_k(p)$ denote the exponent of the prime $p$ in the prime factorization of $k$. Note that $e_k(p)=0$ if $p$ is not a prime factor of $k$.\\
\textbf{Case I:} Suppose that both $m$ and $n$ odd. Then $l$ becomes odd and hence using (\ref{2nd}) in (\ref{sif}) we have
\[\gcdi(h_2,l)=\gcdi(s_1k_1+s_2k_2,l)\gcdi(s_1k_1-s_2k_2,l).\]
Therefore it follows from (\ref{pieq}) that
\[\pideg M_2(\alpha,\beta)=\frac{l}{\gcdi(s_1k_1+s_2k_2,l)}\cdot \frac{l}{\gcdi(s_1k_1-s_2k_2,l)}=\ord(\alpha\beta)\ord(\alpha\beta^{-1}).\]
\textbf{Case II:} Suppose that $e_{m}(2)\neq e_{n}(2)$. For simplicity assume that $e_{m}(2)<e_{n}(2)$. Then by our choice, both $s_2$ and $k_2$ are odd, and $s_1$ is even. This implies that both $s_1k_1\pm s_2k_2$ are odd. Consequently, using (\ref{2nd}) in (\ref{sif}) we have $\gcdi(h_2,l)=\gcdi(s_1k_1+s_2k_2,l)\gcdi(s_1k_1-s_2k_2,l)$ and hence from (\ref{pieq})
\[\pideg M_2(\alpha,\beta)=\ord(\alpha\beta)\ord(\alpha\beta^{-1}).\]
\textbf{Case III:} Suppose that $e_{m}(2)=e_n(2)\geq 1$. Then $l$ becomes even and $s_1,s_2,k_1,k_2$ are odd. 
This implies that one of $(s_1k_1\pm s_2k_2)$ is an odd multiple of $2$, while the other is a multiple of $4$. \\
\textsc{Subcase A:} Suppose that $e_m(2)=e_n(2)=1$. Then by (\ref{cypq}), both $\ord(\alpha\beta)$ and $\ord(\alpha\beta^{-1})$ are odd. Now we simplify (\ref{sif}) as follows
\begin{align*}
    \gcdi(h_2,l)&=\gcdi\left((s_1k_1+s_2k_2)(s_1k_1-s_2k_2),l\right) \\
    &=\begin{cases}
        2\gcdi\left(\frac{s_1k_1+s_2k_2}{2}(s_1k_1-s_2k_2),\frac{l}{2}\right),&\text{if}\ \ e_{s_1k_1+s_2k_2}(2)=1\\
        2\gcdi\left((s_1k_1+s_2k_2)\frac{s_1k_1-s_2k_2}{2},\frac{l}{2}\right),&\text{if}\ \ e_{s_1k_1-s_2k_2}(2)=1
    \end{cases}\\
    &=\begin{cases}
        \gcdi(s_1k_1+s_2k_2,l)\gcdi(s_1k_1-s_2k_2,\frac{l}{2}),&\text{if}\ \ e_{s_1k_1+s_2k_2}(2)=1\\
        \gcdi(s_1k_1-s_2k_2,l)\gcdi(s_1k_1+s_2k_2,\frac{l}{2}),&\text{if}\ \ e_{s_1k_1-s_2k_2}(2)=1
    \end{cases}
\end{align*}
This implies from (\ref{pieq}) that 
\[\pideg M_2(\alpha,\beta)=2\ord(\alpha \beta)\ord(\alpha \beta^{-1}).\]
\textsc{Subcase B:} Suppose that $e_m(2)=e_n(2)>1$. Then by (\ref{cypq}), $\text{ord}(\alpha\beta)$ or $\text{ord}(\alpha\beta^{-1})$ is even with the index of $2$ is less than the index of $2$ in $l$. Now we simplify (\ref{sif}) as follows
\begin{align*}
    \gcdi(h_2,l)&=\gcdi\left((s_1k_1+s_2k_2)(s_1k_1-s_2k_2),l\right) \\
    &=4\gcdi\left(\frac{s_1k_1+s_2k_2}{2}\frac{s_1k_1-s_2k_2}{2},\frac{l}{4}\right)\\
    &=\gcdi\left(s_1k_1+s_2k_2,\frac{l}{2}\right)\gcdi\left(s_1k_1-s_2k_2,\frac{l}{2}\right)
\end{align*}
This implies from (\ref{pieq}) that 
\begin{align*}
    \pideg M_2(\alpha,\beta)&=4\ord((\alpha \beta)^2)\ord((\alpha \beta^{-1})^2)\\
    &=\begin{cases}
        \ord(\alpha\beta)\ord(\alpha\beta^{-1}),& \text{when}\ \ord(\alpha\beta) \ \text{and} \ \ord(\alpha\beta^{-1})\ \text{even}\\  
        2\ord(\alpha\beta)\ord(\alpha\beta^{-1}),& \text{when}\ \{\ord(\alpha\beta),\ord(\alpha\beta^{-1})\}=\{\text{even,\ odd}\}
    \end{cases}
\end{align*}
\begin{rema}
It is important to note that in Case III where $e_2(m) = e_2(n) \geq 1$, if $\ord(\alpha\beta)$ or $\ord(\alpha\beta^{-1})$ is odd, then $\ord(\beta)$ (and $\ord(\alpha)$) does not divide $\ord(\alpha\beta)\ord(\alpha\beta^{-1})$. However, in all other cases, $\ord(\beta)$ (and $\ord(\alpha)$) divides $\ord(\alpha\beta)\ord(\alpha\beta^{-1})$. Moreover, $\ord(\beta)$ and $\ord(\alpha)$ divide simultaneously $\ord(\alpha\beta)\ord(\alpha\beta^{-1})$. 
\end{rema}
Thus combining the above three cases we obtain the following result.
\begin{theo}\label{pithm}
    Let $\alpha,\beta\in \mathbb{K}^*$ such that $\alpha\neq \beta^{\pm 1}$. Then the PI degree of $M_2(\alpha,\beta)$ is 
    \[\pideg M_2(\alpha,\beta)=\begin{cases}
        \ord(\alpha\beta)\ord(\alpha\beta^{-1}),& \text{if}\ \ord(\beta)\divides \ord(\alpha\beta)\ord(\alpha\beta^{-1})\\  
        2\ord(\alpha\beta)\ord(\alpha\beta^{-1}),& \text{otherwise.}
    \end{cases}\]
\end{theo}
\begin{coro}
    In the single-parameter case where $\beta=\alpha^{\pm 1}$, the PI degree of $M_2(\alpha,\alpha^{\pm 1})$ is $m=\ord(\alpha)$. This can be rewritten as \[\pideg M_2(\alpha,\alpha^{\pm 1})=\begin{cases}
        \ord(\alpha^2),&\ \ \text{if} \ \ \ord(\alpha)\divides \ord(\alpha^2)\\
        2\ord(\alpha^2),&\ \ \text{otherwise.}
    \end{cases}\]
    This expression aligns with Theorem \ref{pithm} by taking $\beta=\alpha^{\pm 1}$.
\end{coro}
\begin{coro}
Suppose $m=\ord(\alpha)$ and $n=\ord(\beta)$ such that $\gcd(m,n)=1$, then $\ord(\alpha\beta)=\ord(\alpha\beta^{-1})=l(=mn)$. In this setting, by Theorem \ref{pithm}, the $\pideg M_2(\alpha,\beta)=l^2$. In particular, the Dipper-Donkin quantum matrix algebra (where $\alpha=1,\beta=q$) satisfies this assumption and its PI degree was studied in \cite{ddjj,sbsm}.  
\end{coro}
\section{Construction of Simple Modules over $M_2(\alpha,\beta)$}\label{seccon}
In this section, we construct three different types 
simple $M_2(\alpha,\beta)$-modules when $\alpha\beta\neq 1$. Define \begin{align*}
 l_1:=\begin{cases}
     \ord(\alpha\beta),& \text{if $\ord(\beta)\divides \ord(\alpha\beta)\ord(\alpha\beta^{-1})$}\\
     2\ord(\alpha\beta),& \text{otherwise}
 \end{cases}&\ \ \text{and}\ \ l_2:=\ord(\alpha\beta^{-1}).
\end{align*} 
With the notation above we are ready to define simple $M_2(\alpha,\beta)$-modules.\\
\textbf{Simple modules of type $\mathscr{V}_1(\underline{\mu})$:} For $\underline{\mu}:=(\mu_1,\mu_2,\mu_3,\mu_4)\in \mathbb{({K}^*)}^4$, let us consider the $\mathbb{K}$-vector space $\mathscr{V}_1(\underline{\mu})$ with basis $\{e(a,b):0 \leq a \leq l_1-1, 0 \leq b \leq l_2-1\}$. Define an $M_2(\alpha,\beta)$-module structure on the $\mathbb{K}$-space $\mathscr{V}_1(\underline{\mu})$ as follows: 
\begin{align*}
    e(a,b)X_{11}&=\mu_1 \beta^b e(a\oplus 1,b)\\
    e(a,b)X_{12}&=\begin{cases}
        \mu_2^{-1}\mu_3 (\alpha^{-1}\beta)^b(\alpha\beta)^{-a}e(a,b-1)&\text{if}\ b\neq 0\\
        \mu_2^{-1}\mu_3 \beta^{al_2}(\alpha\beta)^{-a}e(a,l_2-1)&\text{if}\ b=0
    \end{cases}\\
    e(a,b)X_{21}&=\begin{cases}
        \mu_2 e(a,b+1)&\text{if}\ b\neq l_2-1\\
        \mu_2 \beta^{-al_2} e(a,0)&\text{if}\ b=l_2-1
    \end{cases}\\
    e(a,b)X_{22}&=\mu_1^{-1}\alpha^{-b}\left[\mu_4+\beta(\alpha\beta)^{-a} \mu_3\right]e(a\oplus (-1),b)
\end{align*}
where $\oplus$ is the addition modulo $l_1$.
 To establish the well-definedness we need to check that the $\mathbb{K}$-endomorphisms of $\mathscr{V}_1(\underline{\mu})$ defined by the above rules satisfy the defining relations of $M_2(\alpha,\beta)$. Indeed with the above actions we have the following computation:
 \begin{align*}
     e(a,b)(X_{22}X_{11}-X_{11}X_{22})&=e(a,b)X_{22}X_{11}-e(a,b)X_{11}X_{22}\\
    &=(\alpha^{-1}\beta)^b\left[\mu_4+\beta(\alpha\beta)^{-a} \mu_3\right]e(a,b)\\
    &\ \ \ -(\alpha^{-1}\beta)^b\left[\mu_4+\beta(\alpha\beta)^{-a-1} \mu_3\right]e(a,b)\\
     &=\mu_3(\alpha\beta)^{-a}(\alpha^{-1}\beta)^b(\beta-\alpha^{-1})e(a,b)\\
     &=(\beta-\alpha^{-1})e(a,b)X_{12}X_{21}.
 \end{align*}
The other relations are easy to verify. Thus $\mathscr{V}_1(\underline{\mu})$ is an $M_2(\alpha,\beta)$-module.
\par We can easily verify the following with the above actions
\begin{equation}\label{eigen}
  e(a,b)X_{12}X_{21}=\mu_3(\alpha^{-1}\beta)^b(\alpha\beta)^{-a} e(a,b)\ \text{and}\ \ e(a,b)D=\mu_4(\alpha^{-1}\beta)^{b}e(a,b).  
\end{equation}
Thus each $e(a,b)$ is an eigenvector under the actions of $X_{12}X_{21}$ and $D$.
 \begin{theo}\label{f1}
The module $\mathscr{V}_1(\underline{\mu})$ is a simple $M_2(\alpha,\beta)$-module of dimension $l_1l_2$.
\end{theo}
\begin{proof}
Let $\mathscr{W}$ be a non-zero submodule of $\mathscr{V}_1(\underline{\mu})$. We claim that $\mathscr{W}$ contains a basis vector of the form $e(a,b)$. Indeed, any member $w\in \mathscr{W}$ is a finite $\mathbb{K}$-linear combination of such vectors. i.e.,
\[
  w:=\sum_{\text{finite}} \lambda_{ab}~e(a,b)  
\]
for some $\lambda_{ab}\in \mathbb{K}$. Suppose there exist two non-zero coefficients, say, $\lambda_{uv},\lambda_{xy}$. Here $(u,v)$ and $(x,y)$ are distinct pairs with $0\leq u,x \leq l_1-1$ and $0 \leq v,y \leq l_2-1$. Then in particular focus on the two distinct vectors $e(u,v)$ and $e(x,y)$.\\ 
\textbf{Case I:} Assume that $v\neq y$. Then (\ref{eigen}) shows that $e(u,v)$ and $e(x,y)$ are eigenvectors of the operator $D$ corresponding to distinct eigenvalues.\\
\textbf{Case II:} Assume that $v=y$. Since $u\neq x$ by our choice, we now consider two subcases based on $l_1$.\\ 
\textsc{Subcase A:} Suppose $l_1=\ord(\alpha\beta)$. This gives $u \not\equiv x\left(\mo \ord(\alpha\beta)\right)$ and then (\ref{eigen}) shows that $e(u,v)$ and $e(x,y)$ are eigenvectors of the operator $X_{12}X_{21}$ corresponding to distinct eigenvalues.\\
\textsc{Subcase B:} Let $l_1=2\ord(\alpha\beta)$. In this case, $\ord(\beta)$ does not divide $\ord(\alpha\beta)\ord(\alpha\beta^{-1})$. If $u \not\equiv x\left(\mo \ord(\alpha\beta)\right)$, then again using (\ref{eigen}) we show that $e(u,v)$ and $e(x,y)$ are eigenvectors of the operator $X_{12}X_{21}$ corresponding to distinct eigenvalues. Now consider $u \equiv x\left(\mo \ord(\alpha\beta)\right)$. Consequently $u-x=\ord(\alpha\beta)$ as $0\leq u\neq x\leq l_1-1$. Observe that
\[e(a,b)X_{21}^{l_2}=\mu_2^{l_2}\beta^{-al_2}e(a,b).\] This shows that $e(u,v)$ and $e(x,y)$ are eigenvectors of the operator $X_{21}^{l_2}$ corresponding to distinct eigenvalues. Indeed, \[\mu_2^{l_2}\beta^{-ul_2}=\mu_2^{l_2}\beta^{-xl_2} \implies \ord(\beta)\divides(u-x)l_2=\ord(\alpha\beta)\ord(\alpha\beta^{-1}),\] which is a contradiction.
\par Thus any case we show that $e(u,v)$ and $e(x,y)$ are eigenvectors of some operator corresponding to distinct eigenvalues. Using this fact we can reduce the length of the vector $w$ and hence by induction $\mathscr{W}$ contains a basis vector of the form $e(a,b)$. 
Consequently by the actions of the generators $X_{11}, X_{12}, X_{21}$ and $X_{22}$ on $e(a,b)$, we have $\mathscr{W}=\mathscr{V}_1(\underline{\mu})$. Thus $\mathscr{V}_1(\underline{\mu})$ is a simple $M_2(\alpha,\beta)$-module. 
\end{proof}
\textbf{Simple modules of type $\mathscr{V}_2(\underline{\mu})$:} For $\underline{\mu}:=(\mu_1,\mu_2,\mu_3)\in \mathbb{({K}^*)}^3 $, let us consider the $\mathbb{K}$-vector space $\mathscr{V}_2(\underline{\mu})$ with basis $\{e(a,b):0 \leq a \leq l_1-1, 0 \leq b \leq l_2-1\}$. Define an $M_2(\alpha,\beta)$-module structure on the $\mathbb{K}$-space $\mathscr{V}_2(\underline{\mu})$ as follows: 
\begin{align*}
e(a,b)X_{11}&=\begin{cases}
    \mu_1^{-1}\mu_3\alpha^{-1}(\alpha^{-1}\beta)^b[(\alpha\beta)^a-1]e(a-1,b)&\text{if}\ a\neq 0\\
    0& \text{if}\ a= 0
\end{cases}\\
e(a,b)X_{12}&=
    \mu_2^{-1}\mu_3\beta^a(\alpha^{-1}\beta)^be(a,b\dotplus(-1))\\
e(a,b)X_{21}&=\mu_2\alpha^ae(a,b\dotplus 1)\\
e(a,b)X_{22}&=\begin{cases}
    \mu_1e(a+1,b)&\text{if}\ a\neq l_1-1\\
    \mu_1\alpha^{-bl_1}e(0,b)&\text{if}\ a=l_1-1
\end{cases}
\end{align*}
where $\dotplus$ is the addition modulo $l_2$. We can readily confirm that the action above indeed defines an $M_2(\alpha,\beta)$-module structure on $\mathscr{V}_2(\underline{\mu})$, similarly as detailed in the previous case. 
\begin{theo}\label{f2}
The module $\mathscr{V}_2(\underline{\mu})$ is a simple $M_2(\alpha,\beta)$-module of dimension $l_1l_2$.
\end{theo}
\begin{proof} Each vector $e(a,b)$ is an eigenvector under the action of $X_{12}X_{21}, D$ and $X_{21}^{l_2}$ with eigenvalues $\mu_3(\alpha\beta)^a(\alpha^{-1}\beta)^b,\ -\mu_3\alpha^{-1}(\alpha^{-1}\beta)^b$ and $\mu_2^{l_2}\alpha^{al_2}$, respectively. 
With this fact, the proof parallels the proof of Theorem $\ref{f1}$.  
\end{proof}
\textbf{Simple modules of type $\mathscr{V}_3(\underline{\mu})$:} For $\underline{\mu}=(\mu_1,\mu_2)\in (\mathbb{{K}^*})^2$, consider the $\mathbb{K}$-vector space $\mathscr{V}_3(\underline{\mu})$ with basis $\{e(a,b):0 \leq a \leq \ord(\alpha\beta)-1, 0 \leq b \leq l_2-1\}$. Define an $M_2(\alpha, \beta)$-module structure on the $\mathbb{K}$-space $\mathscr{V}_3(\underline{\mu})$ as follows: 
\begin{align*}
e(a,b)X_{11}&=\begin{cases}
    \mu_2\alpha^{-1}(\alpha^{-1}\beta)^b[(\alpha\beta)^a-1]e(a-1,b)&\text{if}\ a\neq 0\\
    0& \text{if}\ a= 0
\end{cases}\\
e(a,b)X_{12}&=
    \mu_1^{-1}\mu_2\beta^a(\alpha^{-1}\beta)^be(a,b\dotplus(-1))\\
e(a,b)X_{21}&=\mu_1\alpha^ae(a,b\dotplus 1)\\
e(a,b)X_{22}&=\begin{cases}
    e(a+1,b)&\text{if}\ a\neq \ord(\alpha\beta)-1\\
    0&\text{if}\ a=\ord(\alpha\beta)-1
\end{cases}
\end{align*}
where $\dotplus$ is the addition modulo $l_2$. We can readily confirm that the action above defines an $M_2(\alpha,\beta)$-module, similarly as detailed in the previous case. 
\begin{theo}\label{f3}
The module $\mathscr{V}_3(\underline{\mu})$ is a simple $M_2(\alpha,\beta)$-module of dimension $\ord(\alpha\beta)\ord(\alpha\beta^{-1})$.
\end{theo}
\begin{proof}
Each vector $e(a,b)$ in $\mathscr{V}_3(\underline{\mu})$ is an eigenvector under the action of $X_{12}X_{21}$ and $D$ with eigenvalues $\mu_2(\alpha\beta)^a(\alpha^{-1}\beta)^b$ and $\ -\mu_2\alpha^{-1}(\alpha^{-1}\beta)^b$, respectively. Given this fact, the proof is similar to Theorem $\ref{f1}$.
\end{proof}
\begin{rema}\label{iso1} We have the following observations :
    \begin{enumerate}
\item $X_{11}^{l_1}$ does not annihilate the simple $M_2(\alpha,\beta)$-module $\mathscr{V}_1(\underline{\mu})$.
\item $X_{11}^{l_1}$ annihilates the simple $M_2(\alpha,\beta)$-module $\mathscr{V}_2(\underline{\mu})$, but $X_{22}^{l_1}$ does not.
\item Both $X_{11}^{l_1}$ and $X_{22}^{l_1}$ annihilate the simple $M_2(\alpha,\beta)$-module $\mathscr{V}_3(\underline{\mu})$.
\end{enumerate}
Thus the above three types of simple $M_2(\alpha,\beta)$-modules are non-isomorphic.
\end{rema}
\section{Classification of simple $M_2(\alpha,\beta)$-modules}\label{seccla} In this section, we fully classify all simple $M_2(\alpha,\beta)$-modules, under the assumption that $\alpha$ and $\beta$ are primitive $m$-th and $n$-th roots of unity, respectively, and that $\mathbb{K}$ is an algebraically closed field. Let $\mathcal{N}$ represent a simple module over $M_2(\alpha,\beta)$. Then by Proposition \ref{pidimresult}, the $\mathbb{K}$-dimension of $\mathcal{N}$ is finite and does not exceed $\pideg\left(M_2(\alpha,\beta)\right)$.
\par When $\alpha\beta=1$, the algebra $M_2(\alpha,\alpha^{-1})$ becomes a quantum affine space $\mathcal{O}_{\Lambda}(\mathbb{K}^4)$ with PI degree $m$, see Subsection \ref{piisoqa}. The simple modules over quantum affine spaces was studied in \cite{smsb}. In this case the $\mathbb{K}$-dimensions of simple $M_2(\alpha,\alpha^{-1})$-module $\mathcal{N}$ are $1$ (when the action of atleast three generators are torsion) or $m$ (when the action of atmost two generators are torsionfree).
\par In the following, assume that $\alpha\beta\neq 1$. The classification of simple $M_2(\alpha,\beta)$-modules relies on the action of specific central or normal elements within the algebra $M_2(\alpha,\beta)$. Notably, the elements $X_{12}, X_{21}$ and $D$ are normal elements. In light of Lemma \ref{normalaction}, the action of any normal element on a simple module $\mathcal{N}$ is trivial or invertible. With this in mind, we now proceed to the following subsections:
\subsection{Simple $X_{12}$-torsion $M_2(\alpha,\beta)$-modules} \label{subclass}
Suppose $\mathcal{N}$ is a $X_{12}$-torsion simple $M_2(\alpha,\beta)$-module. Then $\mathcal{N}X_{12}=0$ and so $\mathcal{N}$ becomes a simple module over the factor algebra $M_2(\alpha,\beta)/\langle X_{12}\rangle$. This factor algebra is isomorphic to a quantum affine space $\mathcal{O}_{\Lambda}(\mathbb{K}^3)$ under the correspondence $X_{11}\mapsto x_1, X_{21}\mapsto x_2, X_{22}\mapsto x_3$, where \[\Lambda=\begin{pmatrix}
   1&\beta^{-1}&1\\
   \beta&1&\alpha^{-1}\\
   1&\alpha&1
\end{pmatrix}.\]
In the roots of unity context, similar to Subsection \ref{subsectionpi}, here the skew-symmetric integral matrix associated with $\Lambda$ has rank $2$ and hence it has a kernel of dimension $1$ and one pair of invariant factors $h_1$, where $h_1=\gcdi(s_1k_1,s_2k_2)$. Consequently by Proposition \ref{mainpi}, the PI degree of $\mathcal{O}_{\Lambda}(\mathbb{K}^3)$ is $l=\lcmu(m,n)$ as $\gcdi(h_1,l)=1$. It is important to note that the simple modules over quantum affine spaces have been classified in \cite{smsb}. This classification provides the structure of $X_{12}$-torsion simple $M_2(\alpha,\beta)$-modules $\mathcal{N}$. In this case, the possible $\mathbb{K}$-dimensions of $\mathcal{N}$ are as follows:
\begin{enumerate}
    \item if the actions of $x_1,x_2,x_3$ on $\mathcal{N}$ are invertible, then $\dime_{\mathbb{K}}(\mathcal{N})=l$.
    \item if the actions of $x_1,x_2$ are invertible and $x_3$ is trivial on $\mathcal{N}$, then $\dime_{\mathbb{K}}(\mathcal{N})=\ord(\beta)$.
    \item if the actions of $x_2,x_3$ are invertible and $x_1$ is trivial on $\mathcal{N}$, then $\dime_{\mathbb{K}}(\mathcal{N})=\ord(\alpha)$.
    \item otherwise $\dime_{\mathbb{K}}(\mathcal{N})=1$.
\end{enumerate}
\subsection{Simple $X_{21}$-torsion $M_2(\alpha,\beta)$-modules} Suppose $\mathcal{N}$ is a $X_{21}$-torsion simple $M_2(\alpha,\beta)$-module. Define a $\mathbb{K}$-linear map $\theta:M_2(\alpha,\beta)\rightarrow M_2(\beta,\alpha)$ by 
\[\theta(X_{11})=X_{11}, \theta(X_{12})=X_{21}, \theta(X_{21})=X_{12}, \theta(X_{22})=X_{22}.\]
We can easily verify that $\theta$ is an isomorphism. With this, $\mathcal{N}$ becomes a $X_{12}$-torsion simple module over $M_2(\beta,\alpha)$. Thus by Subsection \ref{subclass}, we can classify the structure of this type simple modules.
\subsection{Simple $D$-torsion $M_2(\alpha,\beta)$-modules} \label{subclass3}
Suppose $\mathcal{N}$ is a $D$-torsion simple $M_2(\alpha,\beta)$-module. Consequently, $\mathcal{N}D=0$ and so $\mathcal{N}$ becomes a simple module over the factor algebra $M_2(\alpha,\beta)/\langle D\rangle$. This factor algebra is isomorphic to $\mathcal{O}_{\Lambda}(\mathbb{K}^4)/\langle x_{1}x_4-\alpha^{-1}x_2x_3 \rangle$ under the correspondence $X_{11}\mapsto x_1,X_{12}\mapsto x_2, X_{21}\mapsto x_3, X_{22}\mapsto x_4$, where
\[\Lambda=\begin{pmatrix}
    1&\alpha^{-1}&\beta^{-1}&\alpha^{-1}\beta^{-1}\\
    \alpha&1&\alpha\beta^{-1}&\beta^{-1}\\
    \beta&\alpha^{-1}\beta&1&\alpha^{-1}\\
    \alpha\beta&\beta&\alpha&1
\end{pmatrix}.\]
In the roots of unity setting, similar to Subsection \ref{subsectionpi}, here the skew-symmetric integral matrix associated with $\Lambda$ has rank $2$ and one pair of invariant factors $h_1$, where $h_1=\gcdi(s_1k_1,s_2k_2)$. Consequently by Proposition \ref{mainpi}, the PI degree of $\mathcal{O}_{\Lambda}(\mathbb{K}^4)$ is $l=\lcmu(m,n)$ as $\gcdi(h_1,l)=1$. Using the classification of simple modules over quantum affine spaces \cite{smsb}, we can classify the structure of the $D$-torsion simple $M_2(\alpha,\beta)$-modules $\mathcal{N}$. 
In this case, the possible $\mathbb{K}$-dimensions of simple module $\mathcal{N}$ are as follows:
\begin{enumerate}
    \item if the actions of $x_2$ and $x_3$ on $\mathcal{N}$ are invertible, then $\dime_{\mathbb{K}}(\mathcal{N})=l$.
    \item if the action of $x_2$ (or $x_3$) on $\mathcal{N}$ is trivial, then by the relation $x_1x_4=\alpha^{-1}x_2x_3$, the action of $x_1$ or $x_4$ on $\mathcal{N}$ must also be trivial. Consequently we find 
    that the $\dime_{K}(\mathcal{N})$ is $1$ or $\ord(\alpha)$ or $\ord(\beta)$.
\end{enumerate}
In each of the previous three subsections, a simple $M_2(\alpha,\beta)$-module reduces to the structure of a simple module over a much simpler algebra, namely quantum affine space or its quotients.
\subsection{Simple $X_{12},X_{21}$ and $D$-torsionfree $M_2(\alpha,\beta)$-modules} 
Suppose $\mathcal{N}$ is a $X_{12}$, $X_{21}$ and $D$-torsionfree simple $M_2(\alpha,\beta)$-module. It is important to note that the elements \[X_{11}^{l_1}, X_{22}^{l_1}, X_{21}^{l_2}, X_{12}X_{21}\ \text{and}\ D\] are commuting elements in $M_2(\alpha,\beta)$. Since $\mathcal{N}$ is finite-dimensional, these commuting operators have a common eigenvector $v$ in $\mathcal{N}$. So we can take 
\begin{equation}\label{commeigen}
    vX_{11}^{l_1}=\eta_1v, vX_{22}^{l_1}=\eta_2v, vx_{21}^{l_2}=\eta_3v, vX_{12}X_{21}=\eta_4v, vD=\eta_5v
\end{equation}
Note that 
\begin{equation}\label{simplification}
    \eta_5v=vD=v\left(X_{22}X_{11}-\beta X_{12}X_{21}\right)\implies vX_{22}X_{11}=(\eta_5+\eta_4\beta)v.
\end{equation} As $\mathcal{N}$ is $X_{12},X_{21}$ and $D$-torsionfree, the scalars $\eta_3, \eta_4, \eta_5$ are non-zero. In particular, $X_{12}$ and $X_{21}$ act as invertible operators on $\mathcal{N}$. The following cases arise depending on the scalars $\eta_1$ and $\eta_2$.\\
\textbf{Case I:} First assume that $\eta_1 \neq 0$. In this case, the action of $X_{11}$ on $\mathcal{N}$ is invertible as $X_{11}^{l}$ is central and $l_1$ divides $l$. Consequently each of vectors $vX_{11}^aX_{21}^b$ where $0 \leq a \leq l_1-1, 0 \leq b \leq l_2-1$ in $\mathcal{N}$ nonzero. Let us choose
\[\mu_1:=\eta_1^{\frac{1}{l_1}},\ \mu_2:=\eta_3^{\frac{1}{l_2}}, \ \mu_3:=\eta_4,\ \mu_4:=\eta_5\] so that 
$\underline{\mu}=(\mu_1,\mu_2,\mu_3,\mu_4)\in {(\mathbb{K}^*)}^4$.  Now define a $\mathbb{K}$-linear map \[\Phi_1:\mathscr{V}_1(\underline{\mu})\rightarrow \mathcal{N}\] by specifying the image of the basis vectors of $\mathscr{V}_1(\underline{\mu})$ as follows:
\[\Phi_1\left(e(a,b)\right):=\mu_1^{-a}\mu_2^{-b}vX_{11}^aX_{21}^b,\ \ 0\leq a\leq l_1-1, 0 \leq b \leq l_2-1.\] We can easily verify that $\Phi_1$ is a nonzero $M_2(\alpha,\beta)$-module homomorphism. In this verification the following computations, using Lemma \ref{crelation} and the equality (\ref{simplification}), will be very useful:
\begin{align*}
    (vX_{11}^aX_{21}^b)X_{22}&=\begin{cases}
    \alpha^{-b}\left(\eta_5+\beta(\alpha\beta)^{-a}\eta_4\right)(vX_{11}^{a-1}X_{21}^b)&\text{if}\ a\neq 0\\
    \eta_1^{-1}\alpha^{-b}\left(\eta_5+\beta\eta_4\right)(vX_{11}^{l_1-1}X_{21}^b)&\text{if}\ a= 0.
    \end{cases}    
\end{align*}
Thus by Schur's lemma, $\Phi_1$ becomes an $M_2(\alpha,\beta)$-module isomorphism.\\
\textbf{Case II:} Assume that $\eta_1 = 0$ and $\eta_2\neq 0$. In 
Then the action of $X_{11}$ on $\mathcal{N}$ is nilpotent and the action of $X_{21}$ on $\mathcal{N}$ is invertible. Suppose there exists $0\leq r\leq l_1-1$ such that $vX_{11}^{r}\neq 0$ and $vX_{11}^{r+1}=0$. Define $u:=vX_{11}^{r}\neq 0$. Using (\ref{commeigen}), simplify the following
\[uX_{11}=0, uX_{22}^{l_1}=\eta_2u, uX_{21}^{l_2}=\beta^{-rl_2}\eta_3u, uX_{12}X_{21}=(\alpha\beta)^{-r}\eta_4u, uD=-\alpha^{-1}uX_{12}X_{21}.\]
As the actions of $X_{22}$ and $X_{21}$ are invertible, the set $\{uX_{21}^bX_{22}^a:0 \leq a \leq l_1-1, 0 \leq b \leq l_2-1\}$ consists of nonzero vectors of $\mathcal{N}$. Let us choose
\[\mu_1:=\eta_2^{\frac{1}{l_1}},\ \mu_2:=\beta^{-r}\eta_3^{\frac{1}{l_2}}, \ \mu_3:=(\alpha\beta)^{-r}\eta_4 \] so that 
$\underline{\mu}=(\mu_1,\mu_2,\mu_3)\in {(\mathbb{K}^*)}^3$.  Now define a $\mathbb{K}$-linear map \[\Phi_2:\mathscr{V}_2(\underline{\mu})\rightarrow \mathcal{N}\] by specifying the image of the basis vectors of $\mathscr{V}_2(\underline{\mu})$ as follows:
\[\Phi_2\left(e(a,b)\right):=\mu_1^{-a}\mu_2^{-b}uX_{21}^bX_{22}^a,\ \ 0\leq a\leq l_1-1, 0 \leq b \leq l_2-1.\] One can easily verify that $\Phi_2$ is a non zero $M_2(\alpha,\beta)$-module homomorphism. In this verification the following computation, using Lemma \ref{crelation}, will be very useful:
\begin{align*}
(uX_{21}^bX_{22}^a)X_{11}&=\begin{cases}
(\alpha^{-1}\beta)^b\alpha^{-1}[(\alpha\beta)^a-1](uX_{21}^bX_{22}^{a-1})
&\text{if}\ a \neq 0\\
    0,&\text{if}\ a=0
    \end{cases}
\end{align*}
Thus by Schur's lemma, $\Phi_2$ becomes an $M_2(\alpha,\beta)$-module isomorphism.\\
\textbf{Case III:} Finally assume that $\eta_1 = 0$ and $\eta_2= 0$. In this case, the actions of $X_{11}$ and $X_{22}$ on $\mathcal{N}$ are nilpotent. Then the $\mathbb{K}$-space $\ker(X_{11})=\{v\in\mathcal{N}:vX_{11}=0\}$ is nonzero and is invariant under the action of the commuting operators $X_{22}^{l_1}, X_{21}^{l_2}$ and $X_{12}X_{21}$. So there is a common eigenvectors $u\in \ker(X_{11})$ of these commuting operators. Take 
\[uX_{11}=0, uX_{22}^{l_1}=\lambda_1u,\ uX_{21}^{l_2}=\lambda_2u,\  uX_{12}X_{21}=\lambda_3u.\]
Here the scalars $\lambda_2,\lambda_3\in\mathbb{K}^*$ as $\mathcal{N}$ is $X_{12},X_{21}$-torsionfree. The nilpotent action of $X_{22}$ implies $\lambda_1=0$. Suppose $k$ is the smallest integer with $1 \leq k \leq l_1$ such that $uX_{22}^{k-1} \neq 0$ and $uX_{22}^k=0$. We claim that $k=\ord(\alpha\beta)$. In fact simplify the equality $(uX_{22}^k)X_{11}=0$, to obtain 
\begin{align*}
    0=uX_{22}^kX_{11}&=u\left(X_{11}X_{22}^k+\alpha^{-1}[(\alpha\beta)^k-1]X_{12}X_{21}X_{22}^{k-1}\right)\\
    &=\alpha^{-1}[(\alpha\beta)^k-1]\lambda_3(uX_{22}^{k-1}).
\end{align*}
This implies that $(\alpha\beta)^k=1$ for some $1\leq k\leq l_1$ and then by our choice of $l_1$, we have $k=\ord(\alpha\beta)$ or $2\ord(\alpha\beta)$. If $k=2\ord(\alpha\beta)$, then the $\mathbb{K}$-span $S$ of \[\{uX_{21}^bX_{22}^a:\ord{(\alpha\beta)}\leq a\leq 2\ord(\alpha\beta)-1,0\leq b\leq l_2-1\}\] generates a nonzero submodule of $\mathcal{N}$. Since $\mathcal{N}$ is simple, therefore $S=\mathcal{N}$. In particular $u\in S$ and then $uX_{22}^{\ord(\alpha\beta)}=0$, which is a contradiction.
Thus $k=\ord(\alpha\beta)$ and it follows that the vectors $uX_{21}^bX_{22}^a$ where $0 \leq a \leq \ord(\alpha\beta)-1, 0 \leq b \leq l_2-1$ are nonzero. Take $\underline{\mu}:=(\lambda_2^{\frac{1}{l_2}},\lambda_3)\in(\mathbb{K}^*)^2$.
Now define a $\mathbb{K}$-linear map \[\Phi_3:\mathscr{V}_3(\underline{\mu})\rightarrow \mathcal{N}\] by specifying the image of the basis vectors $e(a,b)$ of $\mathscr{V}_3(\underline{\mu})$ as follows:
\[\Phi_3\left(e(a,b)\right):=\mu_1^{-b}uX_{21}^bX_{22}^a,\ \ 0\leq a\leq \ord(\alpha\beta)-1, 0 \leq b \leq l_2-1.\]
One can easily verify that $\Phi_3$ is a non-zero $M_2(\alpha,\beta)$-module homomorphism. Thus by Schur's lemma, $\Phi_3$ becomes an $M_2(\alpha,\beta)$-module isomorphism.
\par The preceding discussion leads to one of the key results of this section, offering a framework for classifying simple 
$M_2(\alpha,\beta)$-modules in terms of scalar parameters.
\begin{theo}\label{iso} 
Suppose $\mathcal{N}$ is a simple $X_{12},X_{21}$ and $D$-torsionfree $M_2(\alpha,\beta)$-module. Then $\mathcal{N}$ is isomorphic to one of the following simple $M_2(\alpha,\beta)$-modules:
\begin{enumerate}
    \item $\mathscr{V}_1(\underline{\mu})$ for some $\underline{\mu}=(\mu_1,\mu_2,\mu_3,\mu_4)\in \mathbb{({K}^*)}^4$ if $\mathcal{N}$ is $X_{11}$-torsionfree.
    \item $\mathscr{V}_2(\underline{\mu})$ for some $\underline{\mu}=(\mu_1,\mu_2,\mu_3)\in \mathbb{({K}^*)}^3$ if $\mathcal{N}$ is $X_{11}$-torsion and $X_{22}$-torsionfree.
    \item $\mathscr{V}_3(\underline{\mu})$ for some $\underline{\mu}=(\mu_1,\mu_2)\in (\mathbb{{K}^*})^2$ if $\mathcal{N}$ is $X_{11},X_{22}$-torsion.
\end{enumerate}
\end{theo}
\begin{coro}\label{corogcd}
The simple modules over $M_2(\alpha,\beta)$ have been studied in \cite{kar} with the deformation parameters $\mathrm{p}$ and $\mathrm{q}$ under the assumption that $\ord(\mathrm{p})=\ord(\mathrm{q})$ is odd, where $\alpha=(\mathrm{pq})^{-1}$ and $\beta=\mathrm{p}\mathrm{q}^{-1}$ in our context. This assumption gives $\ord(\alpha)$ and $\ord(\beta)$ are odd with $\ord(\alpha\beta)=\ord(\alpha\beta^{-1})=t$ (say). In this setting by [Case I, Subsection \ref{subsectionpi}] we have $\pideg M_2(\alpha,\beta)=t^2$ and by Theorem \ref{iso}, the simple $X_{12},X_{21}$, and $D$-torsionfree modules are of dimension $t^2$.  Here the simple modules $\mathscr{V}_1(\underline{\mu}),\mathscr{V}_2(\underline{\mu})$ and $\mathscr{V}_3(\underline{\mu})$ correspond to the \enquote{toroidal}, \enquote{semitoroidal} and \enquote{cylindrical} representations, respectively, as noted in \cite{kar}.
\end{coro}
Finally, we can derive the following result based on the above classification.
\begin{theo}\label{ifof}
    Suppose $\ord(\alpha)=m, \ord(\beta)=n$ and $l=\lcmu(m,n)$.  
    \begin{enumerate}
        \item If $l$ is a proper divisor of $\ord(\alpha\beta)\ord(\alpha\beta^{-1})$, then $\mathcal{N}$ is maximal dimensional simple $M_2(\alpha,\beta)$-module if and only if $\mathcal{N}$ is $X_{12},X_{21}$ and $D$-torsionfree simple module.
        \item If $l$ does not divide $\ord(\alpha\beta)\ord(\alpha\beta^{-1})$, then $\mathcal{N}$ is maximal dimensional simple $M_2(\alpha,\beta)$-module if and only if $\mathcal{N}$ is $X_{ij}$ and $D$-torsionfree simple module for $i,j=1,2$.
    \end{enumerate}
\end{theo}
The proper divisibility hypothesis stated in Theorem \ref{ifof}(1) can not be omitted as given below.
\begin{coro}
    Suppose $\ord(\alpha)=m$ and $\ord(\beta)=n$ such that $\lcmu(m,n)=\ord(\alpha\beta)\ord(\alpha\beta^{-1})$. Such a pair $(m,n)$ exists only if $m=n$. Then
    \begin{enumerate}
        \item each $X_{12},X_{21}$ and $D$-torsionfree simple modules are maximal dimensional.
        \item there are $X_{12}$-torsion or $X_{21}$-torsion or $D$-torsion simple modules of maximal dimensional, as mentioned in Subsections (\ref{subclass})-(\ref{subclass3})).
    \end{enumerate}
\end{coro}
\section{Isomorphism between simple $M_2(\alpha,\beta)$-modules}\label{seciso}
In this section, we investigate the conditions under which two modules, as classified in Theorem \ref{iso}, are isomorphic. Remark \ref{iso1} implies that modules from different types, as described in the theorem, cannot be isomorphic to one another. However, it remains possible for two distinct modules of the same type to be isomorphic.
\begin{theo}\label{iso3}
Let $\underline{\mu}$ and $\underline{\lambda}$ be in $(\mathbb{K}^*)^4$. 
Then $\mathscr{V}_1(\underline{\mu})$ is isomorphic to $\mathscr{V}_1(\underline{\lambda})$ as $M_2(\alpha,\beta)$-modules if and only if
\[\mu_1^{l_1}=\lambda_1^{l_1}\beta^{vl_1},\ \mu_{2}^{l_2}=\lambda_2^{l_2}\beta^{-ul_2}, \ 
        \mu_3=\lambda_3(\alpha\beta)^{-u}(\alpha^{-1}\beta)^v,\ \mu_4=\lambda_4(\alpha^{-1}\beta)^v\]
holds for some $u,v$ such that $0 \leq u \leq l_1-1$ and $0 \leq v \leq l_2-1$.
\end{theo}
\begin{proof}
Let $\psi:\mathscr{V}_1(\underline{\mu}) \rightarrow \mathscr{V}_1(\underline{\lambda})$ be a $M_2(\alpha,\beta)$-module isomorphism. As \[e(a,b)=\mu_1^{-a}\mu_2^{-b}e(0,0)X_{11}^aX_{21}^{b}\] holds in $\mathscr{V}_1(\underline{\mu})$, therefore $\psi$ is uniquely determined by $\psi\left(e(0,0)\right)$. Suppose 
\begin{equation}\label{isoeqn}
\psi\left(e(0,0)\right)=\sum_{u,v}\lambda_{uv}e(u,v)
\end{equation}
where $0 \leq u\leq l_1-1, 0 \leq v \leq l_2-1$ and at least one $\lambda_{uv}\in \mathbb{K}^*$. If possible let there are two non-zero scalars $\lambda_{uv}$ and $\lambda_{xy}$ in (\ref{isoeqn}). The pairs $(u,v)$ and $(x,y)$ are different. Now equating the coefficients of basis vectors on both sides of the equalities 
\[\psi(e(0,0)D)=\psi(e(0,0))D,\ \psi(e(0,0)X_{12}X_{21})=\psi(e(0,0))X_{12}X_{21}\]\[\text{and}\ \ \ \psi(e(0,0)X_{21}^{l_2})=\psi(e(0,0))X_{21}^{l_2}\]
we obtain, respectively,
\[\mu_4=\lambda_4(\alpha^{-1}\beta)^{v}=\lambda_4(\alpha^{-1}\beta)^{y},\ \mu_3=\lambda_3(\alpha\beta)^{-u}(\alpha^{-1}\beta)^v=\lambda_3(\alpha\beta)^{-x}(\alpha^{-1}\beta)^y\]
\[\text{and}\ \ \mu_2^{l_2}=\lambda_2^{l_2}\beta^{-ul_2}=\lambda_2^{l_2}\beta^{-xl_2}.\]
This implies that $v\equiv y~(\mo l_2)$ and then $u\equiv x~(\mo l_1)$. This contradicts the fact that $(u,v)$ and $(x,y)$ are distinct pairs.
Thus $\psi\left(e(0,0)\right)=\lambda_{uv}e(u,v)$ for some $\lambda_{uv}\in \mathbb{K}^*$. This will help us to determine the relationship between the scalars. The actions of $X_{11}^{l_1},\ X_{21}^{l_2}$, $X_{12}X_{21}$ and $D$ under the isomorphism $\psi$ give us 
    \[\mu_1^{l_1}=\lambda_1^{l_1}\beta^{vl_1},\ \mu_{2}^{l_2}=\lambda_2^{l_2}\beta^{-ul_2}, \ \mu_3=\lambda_3(\alpha\beta)^{-u}(\alpha^{-1}\beta)^v,\ \mu_4=\lambda_4(\alpha^{-1}\beta)^v.\]
\par Conversely, assume that the relations between $\underline{\mu}$ and $\underline{\lambda}$ hold. Then define a $\mathbb{K}$-linear map $\phi:\mathscr{V}_1(\underline{\mu}) \rightarrow \mathscr{V}_1(\underline{\lambda})$ by 
\[\phi(e(a,b))=\begin{cases}
(\mu_1^{-1}\lambda_1)^a(\mu_2^{-1}\lambda_2)^b\beta^{av}~e\left(a\oplus u,b\dotplus v\right)&\text{if}\ 0\leq b\leq l_2-v-1\\
(\mu_1^{-1}\lambda_1)^a(\mu_2^{-1}\lambda_2)^b\beta^{av}\beta^{-(u\oplus a)l_2}~e\left(a\oplus u,b\dotplus v\right)&\text{if}\ l_2-v\leq b\leq l_2-1
\end{cases}\]
where $\oplus$ and $\dotplus$ are addition modulo $l_1$ and $l_2$ respectively. Note that $\phi$ defines a linear isomorphism. Finally using the relations between $\underline{\mu}$ and $\underline{\lambda}$, we can verify that $\phi$ is a module isomorphism.
\end{proof}
\begin{theo}\label{2iso}
    Let $\underline{\mu},\underline{\lambda} \in 
    (\mathbb{K}^*)^3$. Then $\mathscr{V}_2(\underline{\mu})$ is isomorphic to $\mathscr{V}_2(\underline{\lambda})$ as $M_2(\alpha,\beta)$-modules if and only if
    \begin{align*}
    \mu_1^{l_1}&=\lambda_1^{l_1}\alpha^{-vl_1}, \ \mu_2^{l_2}=\lambda_2^{l_2}\alpha^{ul_2}, \ 
        \mu_3=\lambda_3(\alpha\beta)^{u}(\alpha^{-1}\beta)^{v}
    \end{align*} holds for some $u,v$ such that $u=0\ \text{or}\ \ord(\alpha\beta)$ and $0 \leq v \leq l_2-1$. 
\end{theo}
\begin{proof}
Suppose $\psi:\mathscr{V}_2(\underline{\mu})\rightarrow \mathscr{V}_2(\underline{\lambda})$ is a module isomorphism. Using an argument similar to the one in Theorem \ref{iso3}, we get $\psi(e(0,0))=\lambda_{uv}e(u,v)$ for some $\lambda_{uv}\in\mathbb{K}^*$. Simplify the equality $\psi(e(0,0)X_{11})=\psi(e(0,0))X_{11}$ to obtain $0=\lambda_{uv}e(u,v)X_{11}$ in $\mathscr{V}_2(\underline{\lambda})$. This implies, by the action of $X_{11}$, that $u=0$ or $u=\ord(\alpha\beta)$. Then the actions of $X_{22}^{l_1}, X_{21}^{l_2}$ and $X_{12}X_{21}$ under $\psi$ provide the required relations between $\underline{\mu}$ and $\underline{\lambda}$. 
\par Conversely assume the relation between $\underline{\mu}$ and $\underline{\lambda}$. Define a $\mathbb{K}$-linear map $\phi:\mathscr{V}_2(\underline{\mu})\rightarrow \mathscr{V}_2(\underline{\lambda})$ by \[\phi(e(a,b))=\begin{cases}
    (\mu_1^{-1}\lambda_1)^{a}(\mu_2^{-1}\lambda_2\alpha^u)^b e(a\oplus u,b\dotplus v)&\text{if}\ 0\leq a\leq l_1-u-1\\
    (\mu_1^{-1}\lambda_1)^{a}(\mu_2^{-1}\lambda_2\alpha^u)^b \alpha^{-(v\dotplus b)l_1} e(a\oplus u,b\dotplus v)&\text{if}\ l_1-u\leq a\leq l_1-1
\end{cases}\]
where $\oplus$ and $\dotplus$ are addition modulo $l_1$ and $l_2$ respectively. Finally we can easily verify that $\phi$ is a module isomorphism.
\end{proof}
\begin{theo}
     Let $\underline{\mu},\underline{\lambda} \in 
    (\mathbb{K}^*)^2$. Then $\mathscr{V}_3(\underline{\mu})$ is isomorphic to $\mathscr{V}_3(\underline{\lambda})$ as $M_2(\alpha,\beta)$-modules if and only if
    \begin{align*}
    \mu_1^{l_2}&=\lambda_1^{l_2},  \ 
        \mu_2=\lambda_2(\alpha^{-1}\beta)^{v}
    \end{align*} holds for some $v$ such that $0 \leq v \leq l_2-1$.
\end{theo}
\begin{proof}
    Same as Theorem \ref{2iso}.
\end{proof}
This concludes the comprehensive classification of simple modules over $M_2(\alpha,\beta)$ up to isomorphism at roots of unity.

\end{document}